\documentclass{amsart}
\newcommand{\floor}[1]{\left\lfloor #1 \right\rfloor}
\def\topp#1{^{(#1)}}
\def\vv#1{{\boldsymbol #1}}
\def\fddto{\stackrel{\rm f.d.d.}{\longrightarrow}}
 \usepackage{tikz}
 \usetikzlibrary{patterns}

 \usepackage[numbers,sort&compress]{natbib}
\usepackage{amsmath,amsthm,amsfonts,amssymb,fancybox,float,url}
 \usepackage{framed}
 \usepackage{color,soul}
 \numberwithin{equation}{section}

\newcommand{\commenthide}[1]{}

\newcommand{\eps}{\varepsilon}

\newcommand{\toD}{\xrightarrow{\mathcal{D}}}

\newcommand{\bin}[2]{\left(\begin{matrix}#1\\#2\end{matrix}\right)}

\newcommand{\calS}{\mathcal{S}}
\newcommand{\calM}{\mathcal{M}}
\newcommand{\calN}{\mathcal{N}}

\newcommand{\calL}{\mathcal{L}}

\newcommand{\RR}{\mathbb{R}}
\newcommand{\NN}{\mathbb{N}}
      \newtheorem{theorem}{Theorem}[section]
       \newtheorem{proposition}[theorem]{Proposition}
       
       \newtheorem{lemma}[theorem]{Lemma}
       \newtheorem{remark}{Remark}[section]
\theoremstyle{definition}

\def\prodd#1#2#3{\prod_{#1=#2}^{#3}}
\def\summ#1#2#3{\sum_{#1=#2}^{#3}}
\def\abs#1{\left|#1\right|}

\def\ccbb#1{\left\{#1\right\}}

\def\pp#1{\left(#1\right)}

\def\floor#1{\left\lfloor #1 \right\rfloor}

\def\vv#1{{\boldsymbol #1}}

\date{Created Monday, Feb 10, 2018. File \jobname.tex. Printed: \today}
\title{Fluctuations of random Motzkin paths}
\author{W\l odzimierz Bryc}
\address
{
W\l odzimierz Bryc\\
Department of Mathematical Sciences\\
University of Cincinnati\\
2815 Commons Way\\
Cincinnati, OH, 45221-0025, USA.
}
\email{wlodzimierz.bryc@uc.edu}

\author{Yizao Wang}
\address
{
Yizao Wang\\
Department of Mathematical Sciences\\
University of Cincinnati\\
2815 Commons Way\\
Cincinnati, OH, 45221-0025, USA.
}
\email{yizao.wang@uc.edu}
\keywords{Motzkin path, scaling limit; phase transition; non-crossing pair partition; Brownian excursion; free Brownian motion; free probability; Laplace transform}
\subjclass[2010]
{60F05; %central limit and other weak theorems
60K35} %Interacting random processes; statistical mechanics type models; percolation theory
\begin{document}\sloppy
\begin{abstract}
It is  known that after scaling a random Motzkin path converges %
in distribution
to a Brownian excursion.
We prove that the fluctuations of the counting processes of the ascent steps, the descent steps and the level steps converge jointly to linear combinations of two independent processes:
a Brownian motion and a Brownian excursion. The  proofs rely on the Laplace transforms and
 an integral representation based on an identity connecting non-crossing pair partitions and joint moments of %
an explicit non-homogeneous Markov process.
\end{abstract}
\maketitle

\section{Introduction}

 Recall that a Motzkin path is a lattice path from $(0,0)$ to $(n,0)$ which does not fall below the  horizontal axis, and uses only the {\em ascents} $(1,1)$, the
 {\em descents} $(1,-1)$ or the {\em level steps} $(1,0)$. (Other authors use terminology ``up steps", ``down steps", and ``horizontal" steps -- here we follow terminology in \citet[page 319]{flajolet09analytic}.)
Each such path is uniquely described by the sequence $\gamma_n=(\eps_1,\dots,\eps_n)$ with $\eps_j\in\{0,\pm 1\}$ which
determine the directions of consecutive steps along the vertical axis.
The cardinality of the set $\calM_n$ of all Motzkin paths, known as the Motzkin number and denoted by $M_n$, %
is   related to the
Catalan numbers $C_k$
 by the formula %
\[%
M_n=\sum_{k=0}^{\lfloor n/2 \rfloor}\bin{n}{2k}C_k, %
\mbox{ where }
C_k = \frac1{k+1}\bin{2k}k.
\]%
 \begin{figure}[H]

  \begin{tikzpicture}[scale=.8]

 \draw[->] (0,0) to (0,3);
 \draw[->] (0,0) to (11,0);
\draw[-,thick] (0,0) to (1,0);
\draw[-,thick] (1,0) to (2,1);
\draw[-,thick] (2,1) to (3,0);
\draw[-,thick] (3,0) to (4,1);
\draw[-,thick] (4,1) to (5,1);
\draw[-,thick] (5,1) to (6,2);
\draw[-,thick] (6,2) to (7,1);
\draw[-,thick] (7,1) to (8,0);
\draw[-,thick] (8,0) to (9,1);
\draw[-,thick] (9,1) to (10,0);

   \node[below] at (1,0) {  $1$};
      \node[below] at (2,0) {  $2$};
       \node[below] at (3,0) {  $3$};
        \node[below] at (4,0) {  $4$};
         \node[below] at (5,0) {  $5$};
          \node[below] at (6,0) {  $6$};
           \node[below] at (7,0) {  $7$};
            \node[below] at (8,0) {  $8$};
             \node[below] at (9,0) {  $9$};
            \node[below] at (10,0) {  $10$};

\end{tikzpicture}

\caption{  \label{Fig1}   Motzkin path $\gamma_{10}=(0,1,-1,1,0,1,-1,-1,1,-1)$, drawn as a linear interpolation.
For this path, the number of level steps
$L_{10}(1)=2$ and the number of ascents $A_{10}(1)=4$.}
\end{figure}
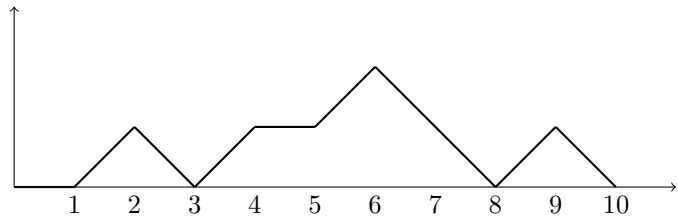

Consider now a random Motzkin path $\gamma_n$  selected uniformly from the set $\calM_n$  of all Motzkin paths of length $n$. Then $(\eps_1,\dots,\eps_n)$
becomes a sequence of random variables with the distribution
that coincides with the distribution of a random walk with independent increments that take values $-1,0,1$ with probability $1/3$ each, conditioned on  staying in the upper
 quadrant and landing at $0$ at time $n$. The general result of \citet{kaigh76invariance} specialized to this setting implies that random process
 $$
\left(\frac{\sqrt{3}}{\sqrt{2n}}\sum_{k=1}^{\lfloor nt\rfloor}\eps_k\right)_{t\in[0,1]}
 $$
converges %
in distribution
 to the Brownian excursion $(B_t^{ex})_{t\in[0,1]}$. %
 Recall that  Brownian excursion is a nonhomogeneous Markov process with explicit transitions which can be interpreted as the
Brownian bridge conditioned to stay strictly positive until time $t=1$.
See for example %
\citep{revuz99continuous,yen13local} and the references therein for more background.

Here we take a closer look, and are in particular interested in
the asymptotic behavior of the three components that constitute a random Motzkin path:
 the counting process $\{A_n(t)\}_{t\in[0,1]}$ of the ascent steps,
 the counting process $\{D_n(t)\}_{t\in[0,1]}$ of  the descent steps, and
  the counting process $\{L_n(t)\}_{t\in[0,1]}$ of  the level steps.
  That is, for each Motzkin path
  $\gamma_n=(\eps_1,\dots,\eps_n)$ %
  write
\[
\eps^+_j = 1_{\ccbb{\eps_j = 1}},\quad \eps^-_j = 1_{\ccbb{\eps_j =-1}},\quad \delta_j  = 1_{\ccbb{\eps_j = 0}}
\]
and
consider three stochastic processes:
$$
A_n(t)=\sum_{k=1}^{\floor{nt}}\eps^+_k,\quad D_n(t)=\sum_{k=1}^{\floor{nt}}\eps^-_k, \quad L_n(t)=\sum_{k=1}^{\floor{nt}} \delta_k.
 $$
See Figure \ref{Fig1} for an illustration.
Clearly, $A_n(t)+D_n(t)+L_n(t)=\floor{nt}$ and the above mentioned consequence of  \citet{kaigh76invariance} can be rephrased as
$$\frac{1}{\sqrt{2n}}\left(A_n(t)-D_n(t)\right)_{t\in [0,1]}\toD \frac{1}{\sqrt{3}}(B_t^{ex})_{t\in[0,1]}$$
as $n\to\infty$.
Here, $\toD$ stands for the weak convergence in $D([0,1])$
with %
Skorohod
 topology.

Our main result is the following  component-wise description of the above convergence.
\begin{theorem}\label{T1}
The finite-dimensional distributions of the $\RR^3$-valued process
$$ \frac{1}{\sqrt{2n}}\left(A_n(t)-\frac{nt}{3}, L_n(t)-\frac{nt}{3},D_n(t)-\frac{nt}{3}\right)_{t\in[0,1]}
$$
converge to the finite-dimensional distributions of
$$\left(\frac{1}{2\sqrt{3}}B_t^{ex}+\frac{1}{6}B_t, -\frac{1}{3}B_t,  \frac{1}{6}B_t-\frac{1}{2\sqrt{3}}B_t^{ex}\right)_{t\in[0,1]}
, $$
where $(B_t)_{t\in[0,1]}$ is a Brownian %
motion, $(B_t^{ex})_{t\in[0,1]}$ is a Brownian excursion, and the   processes $(B_t)_{t\in[0,1]}$ and $(B_t^{ex})_{t\in[0,1]}$ are independent.
\end{theorem}

This
result can be established by different methods. For example, a probabilistic approach is sketched in Section \ref{sec:proba}.
The main purpose of this paper is, by proving Theorem \ref{T1}, however to demonstrate a method that was recently introduced  in our investigation \citep{bryc17asymmetric,bryc17limit} of
asymmetric simple
 exclusion process (ASEP) with open boundary \citep{derrida06matrix}. One of the key ideas therein is to establish an identity connecting the Laplace transform of the statistics of interest,
 essentially the moment generating functions
 of the particles, to the expectation of a functional of a certain inhomogeneous Markov process with explicit transition density functions. The expectation, in the form of an integral representation, makes it possible to compute the asymptotic Laplace transform and then to characterize the limit distribution, although the computation in \citep{bryc17limit} is quite involved.

It is not surprising that the same idea can be applied to the Motzkin-path model considered here, as intrinsic connections between Motzkin paths and ASEP have been well known and explored in earlier research (e.g.~\citep{blythe09continued,brak06combinatorial,brak04asymmetric,corteel07tableaux,corteel11matrix,woelki13parallel}). In particular, the model on Motzkin paths that we considered here is simpler in the sense that the corresponding identity between the generating functions of the counting processes, and the so-called {\em free Brownian motion}, an inhomogeneous Markov process, is more straightforward (see Propositions \ref{P1} and \ref{P3} below) than in the ASEP example. Once the identity is established, the asymptotic limit is then obtained by a straightforward calculation, which is also simpler than in the ASEP example.

 \begin{remark} %
The result in Theorem \ref{T1} itself might be known, although we could not find a reference.
A similar phenomenon as in Theorem \ref{T1} has been described and explained for the steady state of ASEP with open boundary by
 Derrida, Enaud and Lebowitz
 \citep[Section 2.5]{derrida04asymmetric}, where the fluctuations of height functions can also be decomposed into linear combinations of a Brownian motion and a Brownian excursion,
 the two being independent.
\end{remark}

The paper is organized as follows.
In Section \ref{Sect-Sulanke} for pedagogical purposes we  give a simple
integral representation for the generating function of the level steps,    and prove  that finite-dimensional distributions of
  $\frac{1}{\sqrt{2n}}(3L_n(t)-nt)_{t\in[0,1]}$ converge to the finite-dimensional distributions of the Brownian motion $(B_t)_{t\in[0,1]}$.
In Section \ref{Sect:Proofs} we derive a more general integral representation for the joint generating functions that
is needed for the proof of  Theorem \ref{T1}. In Section \ref{Sec:RM} we collect some additional comments and remarks.

\section{Warmup: fluctuations of level steps}\label{Sect-Sulanke}
The  proof that  the
counting process of level steps $(L_n(t))_{0\leq t\leq 1}$  is asymptotically a Brownian motion relies on fewer
technicalities, so we present it separately as an introduction to our approach. The proof of Theorem \ref{T1} presented in
Section \ref{Sect:Proofs} is self-contained and  covers this case.

Recall our notation $\delta_k$ for the indicators of the level steps, and consider the probability generating function
$$
\varphi(\vv u)=\sum_{\gamma_n\in\calM_n}\prod_{j=1}^n u_j^{\delta_j}, %
\quad
\vv{u}=(u_1,\dots,u_n),
$$
for the locations of level steps.
We have the following integral representation for $\varphi(\vv u)$ that uses the Wigner semicircle law %
$(2\pi)^{-1}\sqrt{4-y^2}dy$ supported on $[-2,2]$.
\begin{proposition}\label{P1} For $n=1,2,\dots$,
\begin{equation}
  \label{GenF0}
\varphi(\vv u)=\frac{1}{2\pi}\int_{-2}^2 \prod_{j=1}^n(u_j+y)\sqrt{4-y^2}dy.
\end{equation}
\end{proposition}
\begin{proof}
Each Motzkin path of $n$ steps decomposes uniquely into a Dyck path of $2k$ steps (with ascents and descents only) and $n-2k$ level steps. Thus,
$\gamma_n\in\calM_n$ partitions set $\{1,\dots,n\}$ into the set  $S$ of non-level steps and its complement $S^c$, where the
level steps occur. If cardinality $|S|$ of $S$ is $2k$, then there are %
in total $C_k$ different
 Dyck paths over $S$.  This gives
\[%
\varphi(\vv u)=\sum_{\substack{S\subset\{1,\dots,n\}\\ |S|\in2\NN}} C_{|S|/2}\prod_{j\not\in S}u_j.
\]%
 Since the even moments of the semicircle law are the Catalan numbers, and the odd moments are zero,  see e.g.~ \citet[page 24]{hiai00semicircle}, %
 we can now write this sum over all subsets $S$. We get
$$
\varphi(\vv u)=\sum_{S\subset\{1,\dots,n\}}   \frac{1}{2\pi}\int_{-2}^2 y^{|S|}\sqrt{4-y^2}dy \prod_{j\not\in S}u_j=\frac{1}{2\pi}\int_{-2}^2 \prod_{j=1}^n(u_j+y)\sqrt{4-y^2}dy.
$$
\end{proof}

We can now prove the convergence of the middle component in Theorem \ref{T1}, showing that  $L_n(t)$ behaves just like the sum of independent Bernoulli   random variables
with   probability of success $1/3$.
\begin{proposition}
As $n\to\infty$,
\[
\pp{\frac{3L_n(t)-nt}{\sqrt{2n}}}_{t\in[0,1]}
\fddto
\pp{B_t}_{t\in[0,1]},
\]
where $\fddto$ denotes convergence of %
finite-dimensional
distributions.
\end{proposition}
\begin{proof}
 Fix  $s_0=0<s_1<\dots<s_d<s_{d+1}=1$.
 Since $L_n(0)=0$, it suffices to prove  that the $(d+1)$-dimensional vector of increments
 \[
 \Delta_k\topp n=L_n(s_k)-L_n(s_{k-1}), k=1,\dots,d+1
 \] converges in distribution to
 the corresponding increments of the Brownian motion.

We use formula \eqref{GenF0} to deduce an integral representation for the Laplace transform of
$(\Delta_1,\dots,\Delta_{d+1})$. Denote  $n_k=\floor{ns_k}-\floor{ns_{k-1}}$, starting with $n_1=\floor{ns_1}$ and ending with
$n_{d+1}=n-\floor{n s_d}$.
Splitting the sum into the consecutive blocks
$\calN_k=\left\{j\in\NN: n_{k-1}<j\leq  n_{k}  \right\}$,
 we have
\begin{align}\label{Lap0}
  E \exp\pp{\sum_{k=1}^{d+1} w_k\Delta_k} & = \frac{1}{M_n} \sum_{\gamma\in\calM_n}\exp\pp{\sum_{k=1}^{d+1}w_k\sum_{j\in \calN_k}\delta_j}\\
\nonumber
&=\frac{1}{2\pi M_n}\int_{-2}^2 \prod_{k=1}^{d+1} (e^{w_k}+y)^{n_k}\sqrt{4-y^2}dy.
 \end{align}
For centering, it is more convenient to work with
$$G_n(t):=\frac1{\sqrt{2n}}(3L_n(t)- \floor{nt}),$$
which  is asymptotically equivalent to    $(3L_n(t)- nt)/\sqrt{2n}$.
With
\begin{equation}
  \label{u_{n,k}}
  u_{n,k}=e^{w_k/\sqrt{2n}}
\end{equation}
  we rewrite \eqref{Lap0}   as
\begin{multline}
  \label{Lap1} E \exp\pp{\sum_{k=1}^{d+1} w_k(G_n(s_k)-G_n(s_{k-1}))}
\\ =\frac{1}{2\pi M_n}\int_{-2}^2 \prod_{k=1}^{d+1} \pp{u_{n,k}^2+\frac y{u_{n,k}}}^{n_k}\sqrt{4-y^2}dy.
\end{multline}
  The asymptotic for Motzkin numbers $M_n$ is well known
\begin{equation}\label{M-growth}
  M_n\sim \frac{3^{n+3/2}}{2\sqrt{\pi}n^{3/2}},
\end{equation}
 see e.g.~\citet[Example VI.3 page 396]{flajolet09analytic} who consider $f_n=M_{n-1}$ so their asymptotic expression differs from \eqref{M-growth} by a factor of $3$. %
 Here and below, we write $a_n\sim b_n$ if $\lim_{n\to\infty} a_n/b_n = 1$.

 From now on, we concentrate on the asymptotics of the integral on the right-hand side of \eqref{Lap1}.
The first step is to discard the integral over $y<0$.  Since $w_1,\dots,w_{d+1}$ are fixed, %
we have, for every $k=1,\dots,d+1$,
$u_{n,k}\sim 1$.
If
$-2\leq y<0$ and $1/(1+\delta/2)<u<1+\delta/2$ %
for some    $0<\delta<1$,
then %
$$
\abs{u^2+\frac{y}{u}}\leq \max\ccbb{u^2,\frac 2u}<2+\delta<3
.
$$

So $$ \frac{1}{2\pi M_n}\left|\int_{-2}^0\prod_{k=1}^{d+1}\pp{u_{n,k}^2+\frac y{u_{n,k}}}^{n_k}\sqrt{4-y^2}dy\right|\leq
\frac{(2+\delta)^n}{M_n}\to 0.$$

To determine the asymptotic for the integral over $0<y<2$, mimicking  \citep{bryc17limit}  we substitute $y=2-v^2/(2n)$. We get
\begin{align}\label{f_n}
\frac1{2\pi}&\int_{0}^2 \prod_{k=1}^{d+1} \left(u_{n,k}^2+\frac{y}{u_{n,k}}\right)^{n_k}\sqrt{4-y^2}dy \\
& =\frac1{2\sqrt{2}\pi}\int_{0}^{2\sqrt{n}}\prod_{k=1}^{d+1} \left(u_{n,k}^2+\frac{2}{u_{n,k}}-\frac{v^2}{2nu_{n,k}}\right)^{n_k}\sqrt{4+\frac{v^2}{2n}}\frac{v^2}{n\sqrt{n}}dv \nonumber
\\ & =:\frac{3^n}{2\sqrt{2}\pi n^{3/2}}\int_{0}^{\infty} f_n(v) dv,\nonumber
\end{align}
with
\[
f_n(v) =
1_{\ccbb{v\leq 2\sqrt{n}}}\prod_{k=1}^{d+1} \left(\frac13\pp{u_{n,k}^2+\frac{2}{u_{n,k}}}-\frac{v^2}{6nu_{n,k}}\right)^{n_k}\sqrt{4+\frac{v^2}{2n}}v^2, v\ge0.
\]
We want to show $\lim_{n\to\infty}\int_0^\infty f_n(v)dv = \int_0^\infty\lim_{n\to\infty} f_n(v)dv$.
To do so, we first verify that functions $f_n(v)$ are dominated by an
integrable function.
Since $e^{2x}+2e^{-x}\geq 3$, for any real $w$ and  $0<\delta<1/2$ we can choose $N(w,\delta)$  such that
for   $n\geq N(w,\delta)$ we have
$1/(1+\delta)<e^{w/\sqrt{2n}}<1+\delta$. Then for $0<v^2<4n$ we have
\begin{equation}
  \label{pos0}
  \frac13e^{2w/\sqrt{2n}}+\frac23e^{-w/\sqrt{2n}}-e^{-w/\sqrt{2n}}\frac{v^2}{6n}\geq
1-\frac{2(1+\delta)}{3}>0.
\end{equation}
For   $-1<x\leq 1/2$, we have $ e^{2x}+2e^{-x}\leq 3(1+2 x^2)$. So %
by $1+y\le e^y$ we get %
\[
 \frac13e^{2w/\sqrt{2n}}+\frac23e^{-w/\sqrt{2n}}-e^{-w/\sqrt{2n}}\frac{v^2}{6n}
\leq 1+\frac{w^2}{n}-\frac{1+\delta}{6n} v^2
\leq e^{w^2/n- v^2/(6n)}.
\]
Since the %
left-hand
 side of expression  \eqref{pos0}
 is non-negative,
 by the above bound its %
$n_k$-th
 power is bounded by
$\exp(\frac{n_k}{n}(w^2-v^2/6))$.  Applying this bound to the factors in $f_n(v)$ for a finite number of values of
 $w=w_1,\dots,w_{d+1}$
 and using the fact that $\sum n_k/n =1$
  we see that for large enough $n$
and $v^2\leq 4n$ we have
\begin{equation}\label{f1bound}
0\leq f_n(v)\leq \sqrt{4+\frac{v^2}{2n}} v^2 e^{\max_kw_k^2-v^2/6} \leq \sqrt{6}  e^{\max_kw_k^2}v^2e^{- v^2/6}
,
\end{equation}
 and the latter bound is valid for all $v$ as $f_n(v)=0$ for $v>2\sqrt{n}$.
This bound will justify the use of the dominated convergence theorem below.

It remains to compute the pointwise limit of $f_n(v)$. Recalling
\eqref{u_{n,k}},
we note that
$$\frac13\pp{u_{n,k}^2+\frac{2}{u_{n,k}}}\sim 1+\frac{w_k^2}{2n}+o\pp{\frac1n},
$$
and hence
$$
\lim_{n\to\infty}\left(\frac13\pp{u_{n,k}^2+\frac{2}{u_{n,k}}}-\frac{v^2}{6nu_{n,k}}\right)^{n_k}=
e^{(s_k-s_{k-1})w_k^2/2 -(s_k-s_{k-1}) v^2/6}.
$$
So
$$\lim_{n\to\infty} f_n(v)=
2\exp\pp{\frac12\sum_{k=1}^{d+1}(s_k-s_{k-1})w_k^2} v^2 e^{-{v^2}/6}.
$$
The factor of $2$ arises from $\sqrt{4+\frac{v^2}{2n}}$.
By the dominated convergence theorem,
\begin{align*}
\lim_{n\to\infty}\int_0^\infty f_n(v)dv & =2 %
\exp\pp{\frac12\sum_{k=1}^{d+1}(s_k-s_{k-1})w_k^2} \int_0^\infty v^2 e^{-v^2/6} dv
\\
&=3^{3/2} \sqrt{ 2 \pi}\exp\pp{\frac12\sum_{k=1}^{d+1}(s_k-s_{k-1})w_k^2}.
\end{align*}
So the right-hand side of \eqref{f_n} is asymptotically
$$
\frac{3^{n+3/2}}{2\sqrt{\pi}}\exp\pp{\frac12\sum_{k=1}^{d+1}(s_k-s_{k-1})w_k^2}.
$$
From \eqref{M-growth} we therefore get
$$\lim_{n\to\infty} E \exp\pp{\sum_{k=1}^{d+1} w_k(G_n(s_k)-G_n(s_{k-1}))}=\exp\pp{\frac12\sum_{k=1}^{d+1}(s_k-s_{k-1})w_k^2}
.$$
The right-hand side is the Laplace transform of the increments of a Brownian motion $(B_{s_k}-B_{s_{k-1}})_{k=1,\dots,d+1}$. This ends the proof.
\end{proof}

\section{Proof of Theorem \ref{T1}}\label{Sect:Proofs}

It will be convenient to re-state Theorem \ref{T1} using just two of the processes.
\begin{theorem}
  \label{PT3}
  The finite-dimensional distributions of the process
\begin{equation}
  \label{MultiVariate}
  \frac{1}{{\sqrt{2n}}}\left( 2A_n(t)+L_n(t)-nt,\;  3L_n(t)-nt \right)_{t\in[0,1]}
\end{equation}
  converge to the corresponding finite-dimensional distributions of  $(\frac{1}{\sqrt{3}}B_t^{ex}, B_t)_{t\in[0,1]}$, where
   $(B)_{t\in[0,1]}$ is the Brownian motion,   $(B_t^{ex})_{t\in[0,1]}$ is the Brownian excursion,
  and the two processes are independent.
\end{theorem}
The proof
requires some additional notation and preparation.
 An analog of \eqref{GenF0}
involves a multivariate integral with respect to the finite-dimensional distributions of a   Markov process $(Z_t)_{t\geq
0}$,
which has the univariate distributions $P(Z_t\in dx)=p_t(x)dx$ with
\begin{equation}
  \label{Z-univ} p_t(x)= \frac{\sqrt{4t-x^2}}{2\pi t}1_{\ccbb{|x|\le 2\sqrt t}}
  ,
\end{equation}
and %
 its  transition probabilities for $0\leq s<t$  are given by  $P(Z_t\in dy\mid Z_s=x)=p_{s,t}(x,y)dy$ with
\begin{equation}
  \label{Z-trans}
  p_{s,t}(x,y)=\frac{1}{2\pi} \frac{(t-s)\sqrt{4t-y^2}}{tx^2+sy^2-(s+t)xy+(t-s)^2}\; \mbox{ for $|x|\leq 2\sqrt{s}, |y|\leq 2\sqrt{t}$},
\end{equation}
starting at  $Z_0 = 0$.

The process $(Z_t)_{t\ge 0}$ is known as the {\em free Brownian motion}. See Appendix \ref{A:free-Wick} for explanation.
The joint moments of  $(Z_t)_{t\geq 0}$ are given by a formula that
resembles the formula \citep{isserlis18formula} for the joint moments of the multivariate normal random variable sometimes known as
Wick's theorem. The formula relies on the concept of non-crossing partition introduced by \citet{kreweras72partitions}.
 Recall that a pair partition $\pi$  of $\{1,\dots,d\}$, where $d$ necessarily even, say $d = 2m$,
is a partition into two-element sets  $\{i_1,j_1\},\{i_2,j_2\}\dots,\{i_m,j_m\}$ with $i_k<j_k, k=1,\dots,m$.
   A pair  partition $\pi$ is crossing if there exist two pairs $\{i_k,j_k\},\{i_{k'},j_{k'}\}\in\pi$  such that
   $i_k<i_{k'}<j_k<j_{k'}$, and is noncrossing otherwise.
 Somewhat more generally, we    denote by  $\mathbf{NC}_2(S)$ the set of all non-crossing pair partitions of a
  finite subset $S\subset\NN$ of even cardinality, see \cite[page 132]{nica06lectures}.

The key identity that we need is the following. %
\begin{lemma}\label{L:free-Wick} For $0<t_1\leq \dots\leq t_d $ we have
   \begin{equation} \label{free-Wick}
E(Z_{t_1}Z_{t_2}%
\cdots
 Z_{t_d})=\begin{cases}
   \displaystyle\sum_{\pi\in \mathbf{NC}_2(d)} \prod_{\{i,j\}\in\pi}  t_i , & \mbox{if $d$ is even}
   , \\
   0 &\mbox{ if $d$ is odd}.
\end{cases}
\end{equation}
\end{lemma}
In order not to interrupt the exposition, we postpone the proof to Appendix \ref{A:free-Wick}.
Formula \eqref{free-Wick} then gives the integral formula  for the joint generating function of the
{\em ascent} and {\em level} steps.
\begin{proposition}\label{P3}
For $0<t_1\leq t_2\leq\dots\leq t_n$,
\begin{equation}\label{up2fBM}
\sum_{\gamma_n\in\calM_n} \prod_{j=1}^n t_j^{\eps_j^+}\prod_{j=1}^n u_j^{\delta_j}=E\pp{\prod_{j=1}^n\left(u_j+Z_{t_j}\right)}.
\end{equation}
\end{proposition}
\begin{proof}
We will use a natural bijection from the set of all noncrossing pair partitions on  subsets $S\subset \{1,\dots,n\}$ of even cardinality,  to the set of
Motzkin paths, where a pair $(S,\pi)$ with $\pi\in \mathbf {NC}_2(S)$ is mapped to Motzkin path $\gamma_n=(\eps_1,\dots,\eps_n)$ with $\eps_i=0$ if  $i\not\in S$, $\eps_i=1$ if $\{i,j\}\in \pi$ and
$i<j$ and $\eps_i=-1$ otherwise. This is of course the standard decomposition of a Motzkin path into the level part over $S^c$ and a Dyck path over $S$, the latter
in one-to-one correspondence with noncrossing
pair partitions by \citet[Exercise 6.19]{stanley99enumerative}.
So the left-hand side of \eqref{up2fBM} is
$$\sum_{S\subset\{1,\dots,n\}} \prod_{j\not\in S} u_j \sum_{\pi\in \mathbf{NC}_2(S)} \prod_{\{i,j\}\in \pi} t_i,$$
where the sum is over the subsets $S$ of even cardinality.
But  for nondecreasing  $t_1,\dots,t_n$,
\begin{align*}
\prod_{j\not\in S} u_j\sum_{\pi\in \mathbf{NC}_2(S)} \prod_{\{i,j\}\in \pi} t_i& =
E\left(\prod_{j\not\in S} u_j\prod_{k\in S} Z_{t_k}\right).
\end{align*}
The expectation on the right-hand side is $0$ when $S$ is of odd cardinality, so summing the right-hand side over all $S\subset\{1,\dots,n\}$ we get the right-hand side of \eqref{up2fBM}.
\end{proof}
We remark that   \eqref{up2fBM} is a generalization of \eqref{GenF0}.
 We will use \eqref{up2fBM}
to  prove the convergence of Laplace transforms on an open set that does not include the origin. This will prove the convergence of finite-dimensional distributions
by an application the following result which  is not  well known.
\begin{lemma}\label{L3} Let   $\vv{X}\topp n=(X_1\topp n,X_2\topp n,\dots,X_d\topp n)$   be sequence of  random vectors with Laplace transforms
 $\calL_n(\vv{z})=\calL_n(z_1,\dots,z_d)=E \exp(\sum_{j=1}^d z_j X_j\topp n)$
which  are finite and converge pointwise to a function $\calL(\vv{z})$ for all $\vv{z}$ from an open set in $\mathbb{R}^d$.
If  $\calL(\vv{z})$ is a Laplace transform of a random variable  $\vv{Y}=(Y_1,\dots,Y_d)$, then  $\vv{X}\topp n$ converges in distribution to $\vv{Y}$.
\end{lemma}
In the univariate case, this result is due to   \citet[Section 5.14, page 378, (5.14.8)]{hoffmannjorgensen94probability}. It was rediscovered
by Mukherjea, Rao and Suen \citep[Theorem 2]{mukherjea06note}
 and the proof given there
 works also in the multivariate setting, see \citep[Theorem A.1]{bryc17limit}.
\begin{proof}[Proof of Theorem \ref{PT3}]
Denote
$$F_n(s)=\frac{2A_n(s)+L_n(s)-\floor{ns}}{\sqrt{2n}},G_n(s)= \frac{3L_n(s)-\floor{ns}}{\sqrt{2n}}
,$$
and fix $0=s_0<s_1<s_2<\dots<s_d<s_{d+1}=1$. Since $F_n(0)=G_n(0)=0$, and $(F_n(s),G_n(s))$ differs by at most $2/\sqrt{n}$ from the process in \eqref{MultiVariate},
it is enough to prove that the vector of increments $\vv{X}\topp n\in\RR^{d+1}\times \RR^{d+1}$ with components
\[
X_j\topp n=\pp{F_n(s_j)-F_n(s_{j-1}),G_n(s_j)-G_n(s_{j-1})}, j=1,\dots,d+1,
\]
 converges in distribution to the vector  $\vv{Y}$
with components %
\[
Y_j=\pp{\frac{1}{\sqrt{3}}\pp{B^{ex}_{s_j}-B^{ex}_{s_{j-1}}},B_{s_j}-B_{s_{j-1}}} , j=1,\dots,d+1.
\]
Fix $\vv{z}=(z_1,\dots,z_{d+1})$ with $0<z_1<z_2<\dots<z_{d+1}$, and $\vv{w}=(w_1,\dots,w_{d+1})$. The plan is to
 compute the limit of
the Laplace transforms, $\calL_n(\vv z,\vv w)$ below, and identify the limit as the Laplace transform of  $\vv{Y}$. This will conclude the proof by Lemma \ref{L3}.

For $k=1,\dots, d+1$ it is convenient to introduce the following notation:
$$
\calN_k=\left\{j\in\NN: s_{k-1}n<j\leq  s_{k}n  \right\},\quad n_k=|\calN_k|=\floor{s_kn}-\floor{s_{k-1}n},\; $$
\begin{equation}
  \label{u-and-t}
  u_{n,k}=e^{w_k/\sqrt{2n}}, \quad t_{n,k}=e^{z_k/\sqrt{2n}}.
\end{equation}
We rewrite the Laplace transform as follows
\begin{align*}%
\calL_n(\vv{z},\vv w)& =E \exp \left(\sum_{k=1}^{d+1} z_k (F_n(s_k)-F_n(s_{k-1})) +\sum_{k=1}^{d+1} w_k (G_n(s_k)-G_n(s_{k-1}))\right)\\& =
\prod_{k=1}^{d+1} e^{- n_k (z_k+w_k)/\sqrt{2 n}} E\left(\prod_{k=1}^{d+1} \exp \left(\frac{2 z_k}{\sqrt{2n}} \sum_{j\in\calN_k}\eps_j^+ +\frac{ z_k+3w_k}{\sqrt{2n}} \sum_{j\in\calN_k}\delta_j\right)\right)
\\ &=
\prod_{k=1}^{d+1} t_{n,k}^{-n_k} u_{n,k}^{-n_k}  E\left(\prod_{k=1}^{d+1}  t_{n,k}^{2 \sum_{j\in\calN_k}\eps_j^+}  t_{n,k}^{ \sum_{j\in\calN_k}\delta_j}   u_{n,k}^{3 \sum_{j\in\calN_k}\delta_j} \right),
\end{align*}
where by \eqref{up2fBM} the expectation above is the same as
\[
\frac{1}{M_n} E\left(\prod_{k=1}^{d+1}\pp{t_{n,k}u_{n,k}^3 +Z_{t_{n,k}^2 }}^{n_k}\right).
\]
Therefore the Laplace transform can be written as the   functional of process $(Z_t)_{t\geq 0}$:
\begin{equation}\label{L_n}
\calL_n(\vv{z},\vv w)
=\frac{1}{M_n}E\left(\prod_{k=1}^{d+1}\left(  u_{n,k}^2+\frac{Z_{t_{n,k}^2}}{u_{n,k}t_{n,k}}\right)^{n_k}\right).
\end{equation}
Since asymptotic behavior \eqref{M-growth} for Motzkin numbers is well known, we concentrate on the asymptotic of the integral
on the right-hand side of \eqref{L_n}. Since $t_{n,k}\to 1$ as $n\to\infty$, it is clear that the limit of the Laplace
transforms can be determined from the analysis of the process $(Z_t)_{t\in[1-\eps,1]}$. Moreover, we note that if $Z_t<0$ then
for $\delta>0$ we have
$$\left| u^2+\frac{Z_{t^2}}{ut}\right|\leq \max \ccbb{u^2, \frac2u}<2+\delta$$ for
$u$
close enough  to $1$.
So for %
large enough
 $n>N(\vv{w},\delta)$, if $Z_{t_j^2}<0$ for some $j$ then
\begin{align*}
\left|\prod_{k=1}^{d+1}\left(  u_{n,k}^2+\frac{Z_{t_{n,k}^2}}{u_{n,k}t_{n,k}}\right)^{n_k}\right|
& \leq (2+\delta)^{n_j} \prod_{k\ne j}\left(u_{n,k}^2+\frac2{u_{n,k}}\right)^{n_k}
\\
& \leq
(2+\delta)^{n_j}\left( 3+\delta \right)^{n-n_j}
=(2+\delta)^{n \theta}\left( 3+\delta\right)^{n(1-\theta)}
=C^n,
\end{align*}
where $C=(2+\delta)^{  \theta}\left(3+\delta\right)^{(1-\theta)}\to 3 (2/3)^\theta<3$ as $\delta\to 0$. This shows that for small enough
 $\delta>0$ we have
$C^n/M_n\to 0$. Thus,   only the integral   over positive $Z_{t_k^2}$ contributes to the limit on the right-hand side of \eqref{L_n}.
That is,
\begin{equation}
\label{eq:L_n'}
\calL_n(\vv z,\vv w) \sim
\frac{1}{M_n}E\left(\prod_{k=1}^{d+1}\left(  u_{n,k}^2+\frac{Z_{t_{n,k}^2}}{u_{n,k}t_{n,k}}\right)^{n_k}1_{\ccbb{Z_{t_{n,k}^2}>0,k=1,\dots,d+1}}\right).
\end{equation}

Next, %
\begin{multline*}
E\left(\prod_{k=1}^{d+1}\left(  u_{n,k}^2+\frac{Z_{t_{n,k}^2}}{u_{n,k}t_{n,k}}\right)^{n_k}%
1_{\ccbb{Z_{t_{n,k}^2}>0, k=1,\dots,d+1}}
\right)\\
= \int_0^{2 t_{n,1} }%
\dots\int_0^{2 t_{n,d+1}}\prod_{k=1}^{d+1}\left(u_{n,k}^2+\frac{y_k}{u_{n,k}t_{n,k}}\right)^{n_k}
p_{t_{n,1}^2 }(y_1)\prod_{k=2}^{d+1}p_{t_{n,k-1}^2,t_{n,k}^2}(y_{k-1},y_k)\, d\vv y,
\end{multline*}
where $p_t$ and $p_{s,t}$ are densities from \eqref{Z-univ} and \eqref{Z-trans}. To find the asymptotic behavior of the
latter integral, we substitute  $y_k=t_{n,k} (2-v_k^2/(2n))$ and write
\begin{multline}\label{pre-lim}
E\left(\prod_{k=1}^{d+1}\left(  u_{n,k}^2+\frac{Z_{t_{n,k}^2}}{u_{n,k}t_{n,k}}\right)^{n_k}1_{\ccbb{Z_{t_{n,k}^2,k=1,\dots,d+1}>0}}\right)\\
=\frac{3^{n}}{n^{3/2}}%
\int_0^{2 \sqrt{n}}%
\dots\int_0^{2\sqrt{n}} \prod_{k=1}^{d+1} \left(\frac{u_{n,k}^2}{3 }+\frac{2}{3u_{n,k}}-\frac{ v_k^2}{6nu_{n,k}}\right)^{n_k}
\psi_n(\vv{v}) \,d\vv{v},
\end{multline}
where %
$\vv{v}=(v_1,\dots,v_{d+1})$ and
$$
\psi_n(\vv{v})=\sqrt{n}t_{n,1}v_1p_{{t_{n,1}^2}}(y_1(\vv v))\prod_{k=2}^{d+1} \frac{v_k}{n} t_{n,k}  p_{t_{n,k-1}^2,t_{n,k}^2}(y_{k-1}(\vv v),y_k(\vv v)).
$$

The rest of the proof combines arguments from \citep{bryc17limit,bryc18dual}. For completeness, we include the
details which  in the current setting are more straightforward. Recalling
\eqref{u-and-t},
we first study the limit of the integrand. Clearly,
\begin{equation*}
 \lim_{n\to\infty} \left(\frac{u_{n,k}^2}{3 }+\frac{2}{3u_{n,k}}-\frac{ v_k^2}{6nu_{n,k}}\right)^{n_k}=
 \exp\pp{\frac12(s_k-s_{k-1})w_k^2 - \frac16(s_k-s_{k-1})v_k^2}.
\end{equation*}
Next, we look at the limit of the first factor in $\psi_n(\vv{v})$. We have
$$ \lim_{n\to\infty}\sqrt{n}t_{n,1}v_1p_{t_{n,1}^2 }\left(t_{n,1}\pp{2-\frac{v_1^2}{2n}}\right)=\frac{v_1^2}{\pi\sqrt{2}}.$$
 The remaining factors in $\psi_n(\vv{v})$ also converge. %
 Recalling the definition of $p_{s,t}(x,y)$, we write %
\[
p_{t_{n.k-1}^2,t_{n,k}^2}(y_{k-1}(\vv v),y_k (\vv v)) = \frac1{2\pi}\frac{(t_{n,k}^2-t_{n,k-1}^2)\sqrt{4t_k^2- t_{n,k}^2(2-v_k^2/(2n))^2}}{\varphi((z_k-z_{k-1})/\sqrt{2n},2-v_{k-1}^2/(2n),2-v_{k}^2/(2n))}
,
\]
with
\[
\varphi(\delta,x,y) = e^{-2\delta}\pp{4\sinh^2\delta+(x-y)^2+2xy(1-\cosh\delta)}.
\]
One can show that
\[
\varphi(\varepsilon\delta,2-x\varepsilon^2,2-y\varepsilon^2)\sim \varepsilon^4\pp{(x-y)^2+2(x+y)\delta^2+\delta^4}
\]
as $\varepsilon\downarrow0$.
 See \citep[page 343]{bryc16local}, where our $\varphi$ is the same function as $\varphi_{0,0}$ therein.
Then, it follows that %
\begin{multline}
\label{eq:tangent}
p_{t_{n,k-1}^2,t_{n,k}^2}(y_{k-1}(\vv v),y_k(\vv v)) \\
\sim \frac{4n}\pi\frac{(z_k-z_{k-1})v_k}{(v_{k-1}^2-v_k^2)^2 + 2(v_{k-1}^2+v_k^2)(z_k-z_{k-1})^2+(z_k-z_{k-1})^4}.
\end{multline}
Noting that for $z>0$,
\begin{align*}
&   \frac{2zuv}{(z^2+(u-v)^2)(z^2+(u+v)^2)}= \frac{z/2}{z^2+(u-v)^2}
  -  \frac{z/2}{z^2+(u+v)^2}\\
& \quad\quad= \frac12\int_0^\infty e^{- z x}\cos((u-v)x) \,dx -\frac12\int_0^\infty e^{- z x}\cos((u+v)x) \,dx
\\
& \quad\quad =\int_0^\infty e^{- z x}\sin(ux)\sin(vx)\,dx,
\end{align*}
we obtain
\begin{align*}
 \psi_n &(\vv v) %
 \to
  \frac{v_1^2}{\pi\sqrt 2}\prodd k2{d+1}\frac{4 v_k^2(z_k-z_{k-1})}{\pi((v_k-v_{k-1})^2+  (z_k-z_{k-1})^2)((v_k+v_{k-1})^2+ (z_k-z_{k-1})^2)}\\
& = \frac{2^dv_1v_{d+1}}{\pi^{d+1}\sqrt 2}\prodd k2{d+1} \frac{2 (z_k-z_{k-1})\cdot v_{k-1}v_k}{((v_k-v_{k-1})^2+ (z_k-z_{k-1})^2)((v_k+v_{k-1})^2+  (z_k-z_{k-1})^2)}\\
& = \frac{2^dv_1v_{d+1}}{\pi^{d+1}\sqrt 2}\prodd k1d \int_{\mathbb R_+} e^{-(z_{k+1}-z_{k})x_k}\sin (v_{k+1} x_k)\sin(v_{k}x_k) \, dx_k.
\end{align*}

Next, we verify that we can pass to the limit under the integral sign.
Consider the $k$-th factor in the product. As in the proof of \eqref{f1bound}, %
there
are $N$ and  $\theta>0$ such that for all $n>N$ and all $0<v_k^2<2n$ we have
$$0<\frac{u_{n,k}^2}{3 }+\frac{2}{3u_{n,k}}-\frac{ v_k^2}{6nu_{n,k}}<e^{w_k^2/n}e^{-\theta v_k^2/n}.$$
Thus with $ \theta_1=\theta \min _{1\leq k\leq d+1}(s_k-s_{k-1})/2$ and $n$ large enough (so that $1/n<(s_k-s_{k-1})/2$)   we have
$$
\prod_{k=1}^{d+1}
\left(\frac{u_{n,k}^2}{3 }+\frac{2}{3u_{n,k}}-\frac{ v_k^2}{6nu_{n,k} }\right)^{n_k}\leq %
\exp\pp{\sum_{k=1}^{d+1}w_k^2-\theta_1\sum_{k=1}^{d+1}v_k^2}.
$$
Since the minimum of the denominator in \eqref{Z-trans} occurs for $x=2\sqrt{s}$, $y=2\sqrt{t}$ and
it is then equal $(\sqrt{t}-\sqrt{s})^4$, we see that %
up to a multiplicative constant
 $\psi_n(\vv{v})$ is bounded by %
  $$\prod_{k=1}^{d+1}  \frac{v_k^2(t_{n,k} +t_{n,k-1} )}{n^{3/2}(t_{n,k} -t_{n,k-1} )^3}
  \sim 2^{5d/2} \prod_{k=1}^{d+1}\frac{v_k^2}{(z_k-z_{k-1})^3}. $$
 So the integrand on the right-hand side of \eqref{pre-lim} is bounded by a constant times
 the function
  $v_1^2\cdots v_{d+1}^2\exp(-\theta_1(v_1^2+\dots+v_{d+1}^2))$,
which is integrable over $\RR_+^{d+1}$.
By including the limits of integration as indicators in the integrand, this bound holds for all $0\leq v_1,\dots,v_{d+1}<\infty$.
We can therefore pass to the limit under the integral on the right-hand side of \eqref{pre-lim}. We get
\begin{align*}
\frac1{M_n}E& \left(\prod_{k=1}^{d+1}\left(  u_{n,k}^2+\frac{Z_{t_{n,k}^2}}{u_{n,k}t_{n,k}}\right)^{n_k}1_{\ccbb{Z_{t_{n,k}^2,k=1,\dots,d+1}>0}}\right)\nonumber\\
& \sim \frac1{M_n}\frac{3^n}{n^{3/2}}\int_{\mathbb R_+^{d+1}}\exp\pp{\frac12\summ k1{d+1}(s_k-s_{k-1})w_k^2 - \frac16(s_k-s_{k-1})v_k^2}\nonumber\\
& \quad \quad \times \frac{2^dv_1v_{d+1}}{\pi^{d+1}\sqrt 2}\prodd k1d \int_0^\infty
e^{-(z_{k+1}-z_k)x_k}
\sin(v_{k+1}s_k)\sin(v_kx_k)\, dx_k\, d\vv v.\nonumber\\
\end{align*}
Thus \eqref{eq:L_n'} becomes
\begin{equation}\label{lim-L0}
\lim_{n\to\infty}\calL_n(\vv z,\vv w)  =\frac{2^{d+1/2}}{\pi^{d+1/2}3^{3/2}}\exp\pp{\frac12\summ k1{d+1}(s_k-s_{k-1})w_k^2}\int_{\mathbb R_+^{d+1}}\int_{\mathbb R^d_+} g(\vv v,\vv x)\, d\vv x d\vv v,
\end{equation}
where %
\begin{multline*}%
  g(\vv v, \vv x) \\
  =%
   v_1v_{d+1}e^{-v_{d+1}^2(1-s_d)/6}\prod_{k=1}^{d}e^{-v_k^2(s_k-s_{k-1})/6}
  e^{-(z_{k+1}-z_{k})x_k}\sin (v_{k+1} x_k)\sin(v_{k}x_k).
\end{multline*}

Noting that $s_k-s_{k-1} >0$ and $z_{k+1}-z_k>0$, we see that  $|g(\vv v, \vv x)|$ is bounded by the integrable function of the form
$$  v_1v_{d+1}\exp\pp{-\theta \pp{\sum_{k=1}^{d+1} v_k^2+\sum_{k=1}^d x_k }} $$
for some $\theta>0$.
So  the order of iterated integrals on the right-hand side of
\eqref{lim-L0} can be interchanged.
We then have
\begin{align}
\frac{2^{d+1/2}}{\pi^{d+1/2}3^{3/2}} & \int_{\mathbb R_+^d}\int_{\mathbb R_+^{d+1}}g(\vv v,\vv x)\, d\vv v d\vv x\nonumber =
\int_{\mathbb R_+^d}\frac{\sqrt {8\pi}}{3^{3/2}} \exp\pp{-\summ k1d(z_{k+1}-z_k)x_k}\\
& \quad\times  \frac 1\pi \int_{\mathbb R_+}v_1e^{-s_1v_1^2/6}\sin(v_1x_1)\, dv_1\nonumber\\
& \quad \times \prodd k2d \frac2\pi\int_{\mathbb R_+}e^{-(s_k-s_{k-1})v_k^2/6}\sin(v_kx_k)\sin(v_kx_{k-1})\, dx_k\nonumber\\
& \quad \times \frac 1\pi \int_{\mathbb R_+}v_{d+1}e^{-v_{d+1}^2/6}\sin(v_{d+1}x_d)\, d{v_{d+1}} \, d\vv x.\nonumber
\end{align}
So \eqref{lim-L0} now becomes
\begin{multline}
\lim_{n\to\infty}\calL_n(\vv z,\vv w) \\
 =\exp\pp{\frac12\summ k1{d+1}(s_k-s_{k-1})w_k^2}
 \int_{\mathbb R_+^d}  \exp\pp{-\summ k1d(z_{k+1}-z_k)x_k} f(\vv x)\, d \vv x,
\label{eq:Lap(f)}
\end{multline}
where
\begin{equation}
f(\vv x) = \frac{ \sqrt{8\pi}}{ 3\sqrt{3}}\alpha_{s_1}(x_1)\alpha_{1-s_d}(x_d)   \prod_{k=2}^{d}\beta_{s_k-s_{k-1}}(x_{k-1},x_k) \label{eq:f}
\end{equation}
with
\begin{equation*}
\alpha_{s}(x )=\frac{1}{\pi}\int_0^\infty ve^{-s  v^2/6}\sin(vx )dv = \frac{3 \sqrt{3}}{\sqrt{2\pi} s^{3/2} }  x e^{-3 x^2/(2s)}
,\end{equation*}
and
\begin{align*}
\beta_{s}&(x,y)=\frac{2}{\pi}\int_0^\infty e^{-s v^2/6} \sin (x v)\sin (yv)\, dv
\\
&=
\frac{1}{\pi}\int_0^\infty e^{-s v^2/6} \cos ((y-x) v) d\,v - \frac{1}{\pi}\int_0^\infty e^{-s v^2/6} \cos ((y+x) v) \, dv\\
&=\frac{\sqrt{3}}{\sqrt{2\pi s}}\left(\exp\left(-\frac{3(y-x)^2}{2s}\right)-\exp\left(-\frac{3(y+x)^2}{2s}\right)\right).
\end{align*}

To conclude the proof, we now match the density in \eqref{eq:f} to the joint density of Brownian excursion.
It is known that the joint  probability density function of the Brownian excursion $B^{ex}_{s_1}, \dots, B^{ex}_{s_d}$ is
\begin{equation*}
    f_{s_1,\dots,s_d}(x_1,\dots,x_d) = \sqrt{8\pi}\ell_{s_1}(x_1) \ell_{1-s_d}(x_d)\prod_{k=1}^{d-1}g_{s_{k+1}-s_k}(x_k,x_{k+1})
\end{equation*}
with
\begin{equation*}
  \ell_t(y) = \frac1{\sqrt{2\pi t^3}}  y\exp\left(-\frac{y^2}{2t}\right), {t,y>0}
  ,
\end{equation*}
and
\begin{equation*}
  g_t(y_1,y_2) = \frac1{\sqrt{2\pi t}}\pp{\exp\left(-\frac{(y_1-y_2)^2}{2t}\right) - \exp\left(-\frac{(y_1+y_2)^2}{2t}\right)}, {t,y_1,y_2>0}.
\end{equation*}
See
\citep{durrett77weak}, \citep[page 76]{ito65diffusion},
 or
\citep[page 464]{revuz99continuous}.
Thus the density of $(B^{ex}_{s_1}, \dots, B^{ex}_{s_d})/\sqrt{3}$ is

\begin{align*}
3^{d/2}& f_{s_1,\dots,s_d}\pp{\sqrt{3}x_1,\dots,\sqrt{3}x_d}\\
&=
  \frac{\sqrt{8\pi}}{\sqrt{3}}\cdot \sqrt{3}\ell_{s_1}\pp{\sqrt{3}x_1} \cdot \sqrt{3}\ell_{1-s_d}\pp{\sqrt{3}x_d}\prod_{k=1}^{d-1}\pp{\sqrt{3}g_{s_{k+1}-s_k}\pp{\sqrt{3}x_k,\sqrt{3}x_{k+1}}}
  \\&=\frac{\sqrt{8\pi}}{\sqrt{3}} \frac{\alpha_{s_1}(x_1)}{\sqrt{3}}\frac{ \alpha_{1-s_d}(x_d)}{\sqrt{3}}\prod_{k=2}^{d}\beta_{s_{k}-s_{k-1}}(x_{k-1},x_{k}) = f(\vv x).
\end{align*}
Combining %
\eqref{eq:Lap(f)} and the above, we have shown that
\begin{multline*}
\lim_{n\to\infty}\calL_n(\vv z,\vv w)=
 E \exp \pp{\sum_{k=1}^{d+1} w_k (B_{s_k}-B_{s_{k-1}})} E\exp\pp{\frac1{\sqrt 3}\sum_{k=1}^{d} (z_{k+1}-z_k) B_{s_k}^{ex}%
 }
\\ = E \exp \pp{\sum_{k=1}^{d+1} w_k (B_{s_k}-B_{s_{k-1}})} E\exp\pp{-\frac1{\sqrt 3}\sum_{k=1}^{d+1} z_k (B_{s_k}^{ex}-B_{s_{k-1}}^{ex})%
}.
\end{multline*}
By Lemma \ref{L3}, this ends the proof.
\end{proof}
\begin{remark}
In \citep{bryc17limit}, another ingredient of the proof is to introduce the so-called {\em tangent process}, a positive self-similar Markov process with explicit transition density function that has its own interest, and in particular plays a role in the Laplace transform of Brownian excursion \citep{bryc18dual}.  Here, we choose to not to elaborate on the tangent process in order to reduce the probabilistic flavor of the proof. Instead we only mention that the tangent process arises in the step \eqref{eq:tangent}, where the right-hand side is the same as $2nq_{z_{k-1},z_k}(v_{k-1}^2,v_k^2)$ with $q_{s,t}(x,y)$ being the transition density function of the tangent process as in \citep[Eq.(4.1)]{bryc17limit}.
\end{remark}
\begin{proof}[Proof of Theorem \ref{T1}]
To see that Theorem \ref{PT3}  is an equivalent formulation of Theorem \ref{T1}, note that since $D_n(s)=\floor{ns}-L_n(s)-A_n(s)$, we have
\begin{eqnarray*}
\frac{1}{\sqrt{2n}} \pp{A_n(s)-\frac{\floor{ns}}3} &=& \frac1{2}F_n(s)-\frac{1}{6}G_n(s)
\\
\frac{1}{\sqrt{2n}} \pp{L_n(s)-\frac{\floor{ns}}3}  &=&  \frac{1}{3}G_n(s)
\\
\frac{1}{\sqrt{2n}}\pp{D_n(s)-\frac{\floor{ns}}3} &=&  -\frac1{2}F_n(s)-\frac{1}{6}G_n(s).
\end{eqnarray*}

\end{proof}

\section{Comments and remarks}\label{Sec:RM}
\subsection{Sulanke polynomials}%
The topic of this research was
also
inspired by \citet{sulanke00moment,sulanke01bijective} who studied recursions for  polynomials $\calS_n(t)$ which are the sum over all Motzkin paths of length $n$ of the products of weights along a path.
To define these polynomials, Sulanke assigned weight $1$ to ascent and descent steps, and assigned the weight of indeterminate  $t$ to each level step.
 He then  gave a bijective proof of a recursion for $\calS_n$.
(We note that Sulanke considered elevated Motzkin paths, thus his $f_n(t)$ is $\calS_{n-2}(t)$ in our notation.)

Clearly,
$\calS_n(t)=\varphi(t,t,\dots,t)$, where $\varphi$ is given by \eqref{GenF0}. This gives the generating function for Sulanke polynomials.
\begin{proposition} Setting $\calS_0(t)=1$,  we have
\begin{equation}
  \label{MGF}\sum_{n=0}^\infty z^n \calS_n(t)=\frac{1-tz-\sqrt{(1-tz)^2-4z^2}}{2z^2}.
\end{equation}
\end{proposition}
(When $t=1$ this is of course the well known expression for the generating function of the Motzkin numbers.)
\begin{proof}From \eqref{GenF0} we get
\begin{equation}
  \label{Sulanke2FB}
\calS_n(t)=\frac{1}{2\pi}\int_{-2}^2 (t+y)^n\sqrt{4-y^2}\,dy, \; n=0,1,2\dots
\end{equation}
 Summing the series on the left-hand side of \eqref{MGF} we see that  the generating function of Sulanke polynomials is $G(1/z-t)/z$, where
$$
G(z)=\frac{1}{2\pi}\int_{-2}^2 \frac{\sqrt{4-y^2}}{z-y}\,dy = \frac{z-\sqrt{z^2-4}}{2}
$$
is the   Cauchy--Stieltjes transform of the semicircle law,
see for example \citet[Example 3.1.1]{hiai00semicircle}.
\end{proof}
We remark that asymptotic normality of $L_n(1)$ can also be deduced from \eqref{Sulanke2FB} using the Laplace method, and presumably also from the generating function \eqref{MGF};
 analytic techniques in \citet{flajolet09analytic} are likely to  imply a stronger  local limit law of the Gaussian type.

\subsection{A probabilistic approach for Theorem \ref{T1}}\label{sec:proba}
Here we sketch a probabilistic proof for Theorem \ref{T1}.
Let $$J_n(t)=\floor{nt}-L_n(t)=A_n(t)+D_n(t)$$ denote the number of non-level steps.
Since %
asymptotically
only $1/3$ of the steps of a random walk are horizontal,
we can expect that  $J_n(1)/n\to 2/3$ in probability,  and it is natural to expect that
$\frac1{\sqrt{n}}(J_n(t)-2nt/3)_{t\in[0,1]}$ converges to  $\frac{\sqrt{2}}{3}(B_t)_{t\in[0,1]}$, the Brownian motion scaled by the standard deviation of a Bernoulli random variable with probability of success
$p=2/3$.
Conditionally on $J_n$,  $A_n(t)-D_n(t)=2A_n(t)-J_n(t)$ behaves like a Dyck path on $J_n(1)$ sites,
 which by \citet{kaigh76invariance} converges to the Brownian excursion. So we expect that
$$
\left(\frac{A_n(t)-J_n(t)/2}{\sqrt{J_n(1)}}\right)_{t\in[0,1]}\toD \frac{1}{2} (B_t^{ex})_{t\in[0,1]}.
$$
We can then decompose the process into
$$
\frac{A_n(t)-nt/3}{\sqrt{2n}}=\frac{A_n(t)-J_n(t)/2}{\sqrt{J_n(1)}}\sqrt{\frac{J_n(1)}{2n}}+\frac{J_n(t)-2nt/3}{2\sqrt{2n}}%
,\quad t\in[0,1].
$$
One can show that the two processes on the right-hand side above converge to $\frac1{2\sqrt 3}B^{ex}$ and $\frac 16B$, respectively,
and furthermore the limit of the first
term
 is independent from $(J_n(1))_{n\in\mathbb N}$ as $n\to\infty$. Therefore, the
two limit  processes %
corresponding to
right-hand side above are independent.

\subsection*{Acknowledgement}
  The authors thank an anonymous referee for the careful reading of the manuscript.
  WB's research was supported in part by the Charles Phelps Taft Research Center at the University of Cincinnati. He thanks Jacek Weso\l owski for helpful discussions.
YW's research was supported in part by NSA grant H98230-16-1-0322 and Army Research Laboratory grant W911NF-17-1-0006.

\appendix
\section{Proof of Lemma \ref{L:free-Wick}}\label{A:free-Wick}
 Process $(Z_t)_{t\ge0}$ %
 (\eqref{Z-univ} and \eqref{Z-trans})
 comes from
 \citep[Example 4.9]{bozejko97qGaussian}, see also
 \citep[Example 5.3]{biane98processes}, so Lemma \ref{L:free-Wick} follows from some facts
 from
  free probability. A convenient framework for free probability is the so called
$W^*$-probability space $(\mathcal{A},\tau)$ where $\tau$ is a faithful normal trace on the von Neumann algebra $\mathcal{A}$.
The semicircular  family  %
$(\mathbb X_t)_{t\ge 0}$
 is the set of self-adjoint elements of $\mathcal{A}$ such
that for any choice of $t_1,\dots,t_d\geq 0$ the joint moments are \begin{equation}   \label{mom-free-circ}
\tau(\mathbb{X}_{t_1}\mathbb{X}_{t_2}%
\cdots
 \mathbb{X}_{t_d}) =\sum_{\pi\in
NC_2(d)}\prod_{\{i,j\}\in\pi}\tau(\mathbb{X}_{t_i}\mathbb{X}_{t_j}),
\end{equation}  see \citet[Definition
8.15]{nica06lectures}. The free Brownian motion %
(in free probability)
 is a semicircular family such that
$\tau(\mathbb{X}_{s}\mathbb{X}_{t})=\min\{s,t\}$.

 \citet[page 144]{biane98processes} and %
 Bo\.zejko, K\"ummerer and Speicher
  \citep[Definition 4.1 and Corollary 4.5]{bozejko97qGaussian}  %
  showed
     that for bounded Borel functions $f_1,\dots,f_d$,
the joint moments $E(f_1(Z_{t_1})f_2(Z_{t_2})%
\cdots
 f_d(Z_{t_d}))$  for  $t_1\leq t_2\leq\dots\leq t_d$ coincide with the
corresponding joint moments $\tau(f_1(\mathbb{X}_{t_1})f_2(\mathbb{X}_{t_2})%
\cdots
 f_d(\mathbb{X}_{t_d}))$ of the
 free Brownian motion.
 (For this reason, the process $(Z_t)_{t\ge 0}$ is also referred to as the free Brownian motion, understood as an inhomogeneous Markov process in classical probability theory.)
 Since $|Z_t|\leq 2 \sqrt{t}$ and $\|\mathbb{X}_{t}\|\leq 2\sqrt{t}$, taking  $f_j(x)=x 1_{\ccbb{|x|<2\sqrt{t_j}}}$ we see that
$f_j(Z_{t_j}) =Z_{t_j}$ and  $f_j(\mathbb{X}_{t_j})=\mathbb{X}_{t_j}$, so
 the joint moments on the %
 left-hand side of \eqref{free-Wick}
 coincide with the joint moments of the  semi-circular elements with covariance $\min\{s,t\}$. Thus the right hand side of \eqref{mom-free-circ}  gives the right hand side of \eqref{free-Wick}.

%%%%%%%%%%%%%%%%%%%%%% BBL
\def\cprime{$'$} \def\polhk#1{\setbox0=\hbox{#1}{\ooalign{\hidewidth
  \lower1.5ex\hbox{`}\hidewidth\crcr\unhbox0}}}
  \def\polhk#1{\setbox0=\hbox{#1}{\ooalign{\hidewidth
  \lower1.5ex\hbox{`}\hidewidth\crcr\unhbox0}}}

%%%%%%%%%%%%%%%%%%%%%%%
\end{document}